\documentclass[a4paper]{amsart}
\usepackage{amssymb}
\usepackage{amsfonts}
\usepackage{amsmath}
\usepackage{amsthm}
\usepackage[all]{xy}
\usepackage[english]{babel}
\usepackage{enumerate}%
\providecommand{\U}[1]{\protect\rule{.1in}{.1in}}
\newtheorem{theorem}{Theorem}

\newtheorem*{theoremx}{Theorem}
\newtheorem*{app-clo}{Approximation by closed forms}

\newtheorem*{q}{Question}
\newtheorem*{q1}{Question A}
\newtheorem*{q2}{Question B}
\newtheorem*{ex}{Examples}

\newtheorem{corollary}[theorem]{Corollary}
\newtheorem*{corollaryx}{Corollary}

\newtheorem{proposition}[theorem]{Proposition}
\newtheorem{remark}[theorem]{Remark}

\DeclareMathOperator{\Fol}{Fol}
\DeclareMathOperator{\Dist}{Dist}
\DeclareMathOperator{\id}{id}
\DeclareMathOperator{\dist}{dist}
\DeclareMathOperator{\pt}{pt}
\DeclareMathOperator{\Hom}{Hom}
\DeclareMathOperator{\Pont}{Pont}
\DeclareMathOperator{\Sph}{S}
\DeclareMathOperator{\im}{Im} 
\DeclareMathOperator{\Gr}{Gr} 
\newcommand{\GL}{\mathrm{GL}}

\newcommand{\F}{\mathcal{F}}                
\renewcommand{\d}{\mathrm d}               
\newcommand{\D}{\mathcal{D}}               
\newcommand{\Rr}{\mathbb R}                
\newcommand{\Cc}{\mathbb C}               
\renewcommand{\O}{\mathcal O}               

\begin{document}
\title{A $h$-principle for symplectic foliations}
\author{Rui Loja Fernandes}
\address{Departamento de Matem\'{a}tica\\
Instituto Superior T\'{e}cnico\\1049-001 Lisboa\\ Portugal}
\email{rfern@math.ist.utl.pt}

\author{Pedro Frejlich}
\address{Departamento de Matem\'{a}tica\\
Instituto Superior T\'{e}cnico\\1049-001 Lisboa\\ Portugal}
\email{frejlich@math.ist.utl.pt}

\thanks{RLF is partially supported by the FCT through the Program POCI 2010/FEDER and by project PTDC/MAT/098936/2008. PF is partially supported by FCT doctoral grant SFRH/BD/29087/2006.}

\begin{abstract}
We show that a classical result of Gromov in symplectic geometry extends to the context of symplectic foliations, which we regard as a $h$-principle for (regular) Poisson geometry. Namely, we formulate a sufficient cohomological criterion for a regular bivector to be homotopic to a regular Poisson structure, in the spirit of Haefliger's criterion for homotoping a distribution to a foliation. We give an example to show that this criterion is not too unsharp.
\end{abstract}

\maketitle

\section{Introduction}

The purpose of this paper is to show that the classical h-principle in symplectic geometry due to Gromov can be extended to prove a more general assertion in the context of (regular) symplectic foliations. Let us start by recalling Gromov's theorem:

\begin{theoremx}[Gromov  \cite{G1},\cite{G2}]
Let $V$ be an open manifold, $\xi\in H_{\text{dR}}^2(V;\mathbb{R})$ and $\omega_0$ a non-degenerate $2$-form on $V$. Then there is a path of non-degenerate $2$-forms $\omega_t$, starting at $\omega_0$ and ending at a {\em closed} $2$-form $\omega_1$ with $[\omega_1]=\xi$. 
\end{theoremx}

The result shows that any {\em non-degenerate} 2-form on an open manifold is homotopic (through non-degenerate 2-forms) to a {\em symplectic} form. Hence, on an open manifold the obstruction to the existence of a symplectic form is the same as the obstruction to the existence of a non-degenerate 2-form (a pure topological condition). For closed manifolds the situation is infinitely more complicated. We refer to Gromov's landmark monograph \cite{G2} for details about the h-principle and to \cite{Sp} for a brief historical survey.

It is natural to try to extend Gromov's result to the foliated setting: suppose $V$ is an open manifold of dimension $(2n+q)$, $\mathcal{F}_0$ is a regular, codimension-$q$, smooth foliation on $V$, and $\omega_0$ is a non-degenerate $\mathcal{F}_0$-foliated 2-form.

\begin{q1}
\label{q1}
Is there a path $t\mapsto\omega_t$ consisting of non-degenerate $\mathcal{F}_0$-foliated 2-forms starting at $\omega_0$ and ending at $\mathcal{F}_0$-foliated {\em symplectic} form $\omega_1$?
\end{q1}

Gromov's theorem corresponds to the special case $q=0$. The general case has been studied by M. Bertelson and the answer in general is {\em no}: counterexamples for open manifolds and/or foliations with open leaves are given in \cite{Bth} and a $h$-principle is proved in \cite{B} for foliations whose leaves satisfy a (rather strong) uniformity condition. 

Notice that a foliated symplectic form is the same thing as a (regular) Poisson structure. However, if we take the point of view of Poisson geometry, Question \ref{q1} is not so natural anymore: a Poisson structure on a manifold $V$ is defined by a bivector $\pi$ that satisfies an integrability condition, which can be written using the Schouten bracket as $[\pi,\pi]=0$. Hence, if we are given some bivector $\pi_0$ on $V$, we can reformulate the question above by asking:

\begin{q2}[preliminary version]
Is there a path $t\mapsto\pi_t$ of bivectors on $V$ starting at $\pi_0$ and ending at a {\em Poisson} bivector $\pi_1$?
\end{q2}

This question has a rather trivial answer: $\pi_t:=(1-t)\pi_0$ gives a trivial homotopy to the zero Poisson structure. A much more interesting and natural problem arises if one requires the bivectors to be \emph{regular} (i.e., to have constant rank) and this is precisely the question we are interested here. Given some regular bivector $\pi_0$ on $V$ of rank $n-q$, we ask:

\begin{q2}
Is there a path $t\mapsto\pi_t$ of regular bivectors of rank $n-q$ on $V$ starting at $\pi_0$ and ending at a {\em Poisson} bivector $\pi_1$?
\end{q2}

Note that by ``a homotopy of regular bivectors'' we always mean a homotopy formed by bivectors of the \emph{same} constant rank, so the homotopy $(1-t)\pi_0$, for a non-zero regular bivector $\pi_0$, is excluded. 

If we view a bivector $\pi$ on $V$ as a bundle map $\pi^\sharp:T^*V\to TV$, then any regular bivector defines a distribution, namely its image $\D=\im\pi^\sharp$. Our main theorem states that:

\begin{theorem}
\label{thm:main}
Let $\pi_0$ be a regular bivector on an open manifold $V$. Then there is a path $t\mapsto\pi_t$ of regular bivectors on $V$ starting at $\pi_0$ and ending at a {\em Poisson} bivector $\pi_1$ if and only if the distribution $\D_0=\im\pi_0^\sharp$ is homotopic to an integrable distribution.
\end{theorem}

Note that Gromov's result concerns the special case where the rank equals $\dim V$.

The problem of deforming a distribution to an integrable distribution is, of course, a classical problem. Haefliger in \cite{H} developed an obstruction theory for open manifolds which, in some sense, completely solves the problem. A result of Haeflieger allows one to obtain the following corollary of Theorem \ref {thm:main}:

\begin{corollaryx}
On an open $n$-manifold $V$ satisfying $H^i(V;\mathbb{Z})=0$, for all $i > q+1$, any regular bivector of rank $n-q$ is homotopic, through regular bivectors, to a Poisson structure.
\end{corollaryx}

It should be noted that Theorem \ref{thm:main} also leads to some interesting conclusions concerning Question A, since regular bivectors are the same as distributions carrying a non-degenerate 2-form. In fact, as a corollary we have the following h-principle for symplectic foliations:

\begin{theorem}
\label{thm:main:alt}
Given a foliation $\mathcal{F}_0$ of an open manifold $V$ and a $\mathcal{F}_0$-foliated non-degenerate 2-form $\omega_0$, there is a path $t\mapsto (\mathcal{F}_t,\omega_t)$, consisting of a foliation $\mathcal{F}_t$ and a $\mathcal{F}_t$-foliated non-degenerate 2-form $\omega_t$, starting at $(\mathcal{F}_0,\omega_0)$ and ending at a $\mathcal{F}_1$-foliated {\em symplectic} form $\omega_1$.
\end{theorem}

Actually the proof of Theorem \ref{thm:main} can be reduced to the proof of Theorem \ref{thm:main:alt}, and this follows by slightly modifying a proof of Gromov's classical result due to Eliashberg and Mishachev \cite{EM}.

In conclusion, we have: 
\begin{enumerate}
 \item even on an open manifold, in general one cannot homotope a foliated non-degenerate 2-form (through such structures) to a foliated symplectic structure if we insist that the underlying foliation remains fixed, and yet
 \item on an open manifold this homotopy can always be found if we allow our foliation to be deformed as well. 
\end{enumerate}

This paper is organized as follows. In Section \ref {sec:setting}, we briefly recall some basic concepts and results. In Section \ref{sec:h:principle}, we prove the h-principle for symplectic foliations (Theorem \ref{thm:main:alt}). Section \ref{sec:obstr} is dedicated to Haeflieger's obstruction theory and we show that the criterion (Theorem \ref{thm:main}) for integrability for regular bivectors holds. In the closing section we modify an example of Bott to construct a regular bivector on an open manifold which is not homotopic to a regular Poisson structure.

\section{The Setting}
\label{sec:setting}

Let us first establish some basic notations and definitions. All manifolds are assumed to be smooth and to carry suitable Riemannian metrics.

Throughout this paper $V$ will denote an  {\bf open} $n$-manifold, by which we mean that $V$ admits a proper, positive, Morse function $f:V \longrightarrow\mathbb{R}$ with no critical points of maximal index $n$. Equivalently, $V$ is a smooth manifold, none of whose connected components is a closed manifold. The particular feature of open manifolds that will be of interest to us (cf. \cite[Sec. 4.3]{EM}, \cite{G2}) is that they admit a smooth triangulation inside which there exists a subcomplex $V_0\subset V$ (a {\bf subpolyhedron} of $V$ by definition) of positive codimension such that, for an arbitrarily small open neighbourhood $U$ of $V_0$, there is a diffeotopy of  $\id_V$ to a map $V\longrightarrow U$ which is fixed on $V_0$; by this we mean that there is a one-parameter family of open embeddings
\[
g_t:V\longrightarrow V
\]
with $g_0=\id_V$, $g_1(V)\subset U$ and $g_t\vert_{V_0}=\id_{V_0}$. The family $g_t$ is called a {\bf compression} of $V$ into $U$. In fact, note that the Morse CW-complex $V(f)\subset V$ associated to $f$ has positive codimension by the absence of of local maxima, and thus the gradient flow of $f$ compresses $V$ into an arbitrarily small neighbourhood of $V_0:=V(f)$.

Any subpolyhedron $A$ of $V$ of positive codimension which admits arbitrarily small neighbourhoods into which $V$ can be compressed will be called a {\bf core} of $V$. The Morse CW-complex $V(f)$ of the Morse function $f$ described above  (in which no $n$-handles occur) is a core for $V$.

\begin{ex}
Take $V=\Sph^n-\pt$. Then any point in $V$ is a core. $Z_1\vee Z_2$ is a core for $\Sph^1\times \Sph^1-\pt$, where $Z_1,Z_2$ are generators for $\pi_1(\Sph^1\times \Sph^1)$.
\end{ex}

Throughout the text by a smooth fibration $p:X\to V$ we will mean a locally trivial smooth fibre bundle. Associated to it we have the the jet fibrations 
\[ p^r:X^{(r)}\to V\] 
and its associated fibrations (see, e.g., \cite[Sec. 2.4]{Hr}):
\begin{align*}
 p^r_{s} :X^{(r)} &\longrightarrow X^{(s)} \text{ for  } 0\le s\leq r, \quad 
X^{(0)}=X \text,\\
p^r_t&=p^r_s\circ p^s_t \text{ for  } t\leq s\leq r
\end{align*}
together with a sequence of jet maps at the level of its sections:
\begin{align*}
 j^r:\Gamma(V,X)&\longrightarrow\Gamma(V,X^{(r)}), \\
 \quad p^r_{s}\circ j^r &= j^s .
\end{align*}
Elements in the image of $j^r$ are called {\bf holonomic} sections of $p^r:X^{(r)}\to V$; these form a very thin subspace inside $\Gamma(V,X^{(r)})$. So it is all the more remarkable that the following theorem of Eliashberg and Mischachev holds: 

\begin{theorem}[Holonomic Approximation \cite{EM} ]
Given a polyhedron $A \subset V$ of positive codimension, $U_0$ an open neighbourhood of $A$, a section $F\in\Gamma(U_0,X^{(r)})$, and two positive functions $\varepsilon,\delta: V\longrightarrow\mathbb{R}_{+}$, there exist
\begin{enumerate}
  \item[(i)] An isotopy $h^t:V\longrightarrow V$ which is $\delta$-small in the $C^0$-sense -- i.e., it is such that 
\[
\dist(y,h^t(y))<\delta(y) \text{  for all  } t {\text ;}
\]
  \item[(ii)]  A neighbourhood $U_1\subset U_0$ of $h^1(A)$;
  \item[(iii)]  A section $f\in\Gamma(U_1,X^{(r)})$ whose $r$-jet is $\varepsilon$-close to $F$ over $U_1$  -- i.e., such that 
\[
\dist(F(y),j^rf(y))<\varepsilon(y) \text{ for all } y\in U_1.
\]
\end{enumerate}
\end{theorem}

Next observe that the exterior derivative
\[
\d:\Gamma(V,\wedge^{p-1}T^{\ast}V)\longrightarrow\Gamma(V,\wedge^{p}T^{\ast}V)
\]
factors as $\d=D\circ j^1$, where $D$ is the mapping (at the level of sections) induced by the symbol of $\d$
\[
D:\left( \wedge^{p-1}T^{\ast}V\right) ^{(1)}\longrightarrow \wedge^{p}T^{\ast}V
\]
which is an affine fibration. This means that given a $p$-form $\omega\in\Gamma(V,\wedge^{p}T^{\ast}V)$, there is a section $F_{\omega}\in\Gamma(V,\left( \wedge^{p-1}T^{\ast}V\right) ^{(1)})$, called a {\bf formal primitive} of $\omega$, projecting to it : $DF_{\omega}=\omega$. This formal primitive is unique up to homotopy since the fibration has contractible fibres; furthermore, given any $(p-1)$-form $\alpha\in\Gamma(V,\wedge^{p-1}T^{\ast}V)$ one can choose $p^1_0\circ F_{\omega}=\alpha$.

From Holonomic Approximation and the above remarks, one deduces the following result (see \cite[Thms. 4.7.1 and 4.7.2]{EM}), which is actually what we need:

\begin{theorem}[Approximation by closed forms]
\label{thm:approx}
Let $\omega$ be a $p$-form on the open manifold $V$. Then given a positive function
\[
\varepsilon:V\longrightarrow\mathbb{R}
\]
and a cohomology class $\xi\in H_{\text dR}^p(V;\mathbb{R})$, there exist : (1) a core $A$ of $V$; (2) a $p$-form $\phi\in\xi$ which is $\varepsilon$-close to $\omega\vert_{U}$ on some open neighborhood $U$ of $A$.
\end{theorem}

\section{$h$-principle for symplectic foliations}
\label{sec:h:principle}
In this section we will prove the h-principle for symplectic foliations on open manifolds (Theorem \ref{thm:main:alt}).

Given a foliation $\F$ on a manifold $V$, by an $\F$-foliated form $\omega$ of degree $p$ we mean a section $\omega\in \Gamma(V,\wedge^pT^{\ast}\mathcal{F})$. We can always extend $\omega$ to a $p$-form $\widetilde{\omega}\in \Gamma(V,\wedge^pT^{\ast}V)$ (not uniquely) and often we will not distinguish between $\omega$ and an extension $\widetilde{\omega}$. The de Rham differential induces a differential $\d_\F$ on foliated forms, the leafwise de Rham differential, which is such that a form $\omega$ is $\d_\F$-closed if and only if $\d\widetilde{\omega}$ pulls back to zero on each leaf of $\F$.

A pair $(\F,\omega)$ where $\omega$ is an $\F$-foliated, non-degenerate, 2-form will be called an \textbf{almost symplectic foliation}. If, additionally, $\omega$ is $\d_\F$-closed then we call $(\F,\omega)$ a  \textbf{symplectic foliation}.

\begin{theorem}
\label{thm:hprinciple:symp:fol}
Let $V$ be an open manifold of dimension, let ($\mathcal{F}_0$,$\omega_0$) be a codimension-$q$ almost symplectic foliation and fix a cohomology class $\xi\in H_{\text{dR}}^2(V;\mathbb{R})$. Then there is a homotopy $(\mathcal{F}_t,\omega_t)$ of almost-symplectic, codimension-$q$ foliations, such that:
\begin{enumerate}
\item[(i)]  $(\mathcal{F}_1,\omega_1)$ is a {\em symplectic} foliation and 
\item[(ii)]  $\omega_1$ can be represented by a global closed 2-form lying in $\xi$.
\end{enumerate}
The homotopy $(\mathcal{F}_t,\omega_t)$ can be taken to be smooth in $t$.
\end{theorem}

\begin{proof}
A codimension-$q$ distribution on $V$ is simply a section of the Grassmannian bundle $\Gr_{2n}(TV)$, so we can topologize the space $\Dist_q{V}$ of codimension-$q$ distributions on $V$ with the compact-open topology. We also topologize $\Gamma(V,\wedge^2T^{\ast}V)$ with the compact-open topology and we consider the subspace topology on the space of almost symplectic, codimension-$q$ distributions :
\begin{align*}
 \Delta_q\subset\Dist_q(V)\times\Gamma(V,\wedge^2T^{\ast}V),\\ 
 \Delta_q=\left\lbrace (\D,\omega):\left( \iota_\D^{\ast}\omega\right)^n\neq 0 \right\rbrace.
\end{align*}
On the space $\Phi_q$ of almost symplectic, codimension-$q$ foliations and on the space $\Omega_q$ of symplectic, codimension-$q$ foliations we take also the induced topologies :
\begin{align*}
 \Phi_q&:= \Fol_q(V)\times\Gamma(V,\wedge^2T^{\ast}V)\subset\Delta_q,\\
 \Omega_q&:= \left\lbrace (\mathcal{F},\omega)\in\Phi_q:\omega\text{ is }\mathcal{F}-\text{symplectic} \right\rbrace\subset \Phi_q.
\end{align*}
Thus our initial data $(\mathcal{F}_0,\omega_0)$ lives in $\Phi_q$.

Let us observe that, since non-degeneracy of a form is an open condition, there exists a positive function
\[
\varepsilon:V\longrightarrow\mathbb{R}
\]
such that
\[
\dist(\omega(y),\omega_0(y))<\varepsilon(y),\, \forall y\in V \quad\Longrightarrow \omega\text{  is }\mathcal{F}_0\text{-leafwise non-degenerate.}
\]
Hence, applying approximation by closed forms (Theorem \ref{thm:approx}) we can choose some $\rho >0$ and a closed form $\phi$ such that $[\phi]=\xi$ and which is $\varepsilon$-close to $\omega_0$ on the $\rho$-neighbourhood $U_{\rho}$ of a core $A$ of $V$.  

Next, we choose a smooth function $\chi:V\longrightarrow [0,1]$, which is identically zero outside $U_{\rho}$ and identically $1$ on $U_{\rho/2}$, and we define a homotopy of 2-forms by setting:
\[
\omega : [0,1/2]\longrightarrow\Gamma(V,\wedge^2T^{\ast}V)
\]
\[
t\mapsto\omega_0+2t(\phi-\omega_0)\chi.
\]
Then $\omega$ is a continuous map such that :
\begin{enumerate}
\item[(i)]  $\omega(0)=\omega_0;$
\item[(ii)]  $\omega(t)$ is $\mathcal{F}_0$-leafwise non-degenerate for all $t\in[0,1/2]$;
\item[(iii)]  $\omega(1/2)$ is closed on $U_{\rho/2}$.
\end{enumerate}
Hence, if for $t\in[0,1/2]$ we let $\mathcal{F}(t)$ be the stationary homotopy at $\mathcal{F}_0$, then $t \mapsto(\mathcal{F}(t),\omega(t))$ takes values in $\Phi_q$.

In order to define the second half of the homotopy, we choose a compression $g_t:V\to V$ between $g_0=1_V$ and $g_1:V\to U_{\rho/2}$ and we define a continuous path $(\mathcal{F}(t),\omega(t))\in \Phi_q$ for $t\in [1/2,0]$ by setting:
\[
t\mapsto((g_{2t-1})^{\ast}\mathcal{F}(1/2),(g_{2t-1})^{\ast}\omega(1/2))\qquad (t\in[1/2,1]).
\]
The concatenated homotopy
\[
t\mapsto(\mathcal{F}(t),\omega(t)) \qquad (t\in[0,1]),
\]
is the one we sought, since $(\mathcal{F}(1),\omega(1))$ is a symplectic foliation and the 2-form $\omega(1)$ lies in $g_1^*\xi=\xi$.

Note that the homotopy is only continuous in $t$, but a standard argument involving reparameterization with vanishing derivatives at the end points makes it smooth.
\end{proof}

\section{Obstructions to integrability}
\label{sec:obstr}

Recall that $\Delta_q$ denotes the space of almost symplectic, codimension-$q$ distributions and that $\Omega_q$ denotes the space of symplectic, codimension-$q$ foliations on the manifold $V$. In this section, we wish to address the following

\begin{q}
Are there strictly topological conditions that one can impose upon $V$ so as to ensure that $\pi_0(\Omega_q)\longrightarrow\pi_0(\Delta_q)$ is an isomorphism ?
\end{q}

Notice that the previous theorem implies that $\Omega_q\hookrightarrow\Phi_q$ induces an isomorphism at the level of $\pi_0$. To handle the map
\[
\pi_0(\Phi_q)\longrightarrow\pi_0(\Delta_q)
\]
we will invoke A. Haefliger obstruction theory (see \cite{H}). 

For any topological groupoid $\mathcal{G}$ we will denote by $B\mathcal{G}$ its classifying space. Also, we let $\Gamma_q$ denote the topological groupoid of germs of local diffeomorphisms of $\Rr^q$ fixing the origin. The operation of taking the differential at the origin gives a homomorphism of topological groupoids 
\[ \d: \Gamma_q\to \GL_q\]
and, hence, a map at the level of the classifying spaces:
\[
\nu:B\Gamma_q\longrightarrow B\GL_q.
 \]
This map is the key to solve the problem of deforming a distribution to an integrable one: 

\begin{theorem}[Haefliger \cite{H}]
Let $\D_0$ be a a codimension-$q$ distribution in an open $n$-manifold $V$ and let $\tau:V\to B\GL_n$, $\xi:V\to B\GL_{n-q}$ and $\alpha:V\to B\GL_q$ denote the classifying maps for the bundles $TV$, $\D_0$ and $TV/\D_0$, respectively. Then there exists a homotopy $\D_t$ of distributions starting at $\D_0$ and ending at a foliation $\D_1=T\F$ if and only if the following commutative diagram can be solved :
\[
 \xymatrix{
 & &  B\Gamma_q\times B\GL_{n-q} \ar[d]^{\nu\times 1}\\
 & & B\GL_q\times B\GL_{n-q} \ar[d]^{\oplus}\\
V \ar@{-->}[uurr] \ar[urr]_{\xi\times\alpha} \ar[rr]_{\tau} & & B\GL_n
}
\]
where $\oplus$ denotes the arrow (in the homotopy category {\sf hTop}) inducing direct sum of vector bundles. 
\end{theorem}

This result shows that the integrability problem for distributions on open manifolds can be completely reduced to one in obstruction theory. Moreover, it can be shown that the map $\nu:B\Gamma_q\longrightarrow B\GL_q$ is $(q+1)$-connected, hence:

\begin{corollary}
On an open $n$-manifold $V$ with $H^i(V;\mathbb{Z})=0$ for all $i > q+1$, any codimension-$q$ distribution is homotopic to a foliation.
\end{corollary}

In order to apply this to our setting, let us spell out what Haefliger's theorem says: given a distribution $\D_0$ such that the diagram above can be solved one can find a path of bundle isomorphisms
\[
\varphi_t:TV\longrightarrow \D_0\oplus (\D_0)^{\perp}
\]
where $\varphi_1$ maps $T\mathcal{F}$ onto $\D_0$ for some codimension-$q$ foliation $\F$. That is, the path
\[
t\mapsto\varphi_t^{\ast}\D_0=:\D_t
\]
ends at $T\mathcal{F}$. If an almost-symplectic $\D_0$-form $\omega_0$ has been provided, one can transport this 2-form along the path $\varphi_t$ by setting
\[
\omega_t:=\varphi_t^{\ast}\omega_0
\]
so that, at all times we have $(\D_t,\omega_t)\in\Delta_q$ and at the end-point $(\D_1,\omega_1)\in\Phi_q$. Now we can use Theorem \ref{thm:hprinciple:symp:fol}
to construct a continuous path starting at $(\D_1,\omega_1)$ and ending at some element in $\Omega_q$. This yields:

\begin{proposition}
\label{prop:main}
Given a distribution $\D_0$ on an open manifold $V$ with a non-degenerate 2-form $\omega_0$, there is a homotopy $(\D_t,\omega_t)\in\Delta_q$ starting at $(\D_0,\omega_0)$ and ending at a $\mathcal{F}_1$-foliated {\em symplectic} form $\omega_1$ if and only if the distribution $\D_0$ is homotopic to a foliation.
\end{proposition}

Let us recast the previous result in terms of Poisson geometry. Any bivector $\pi\in\Gamma(V,\wedge^2TV)$ defines a bundle map
\[
\pi ^\sharp: T^{\ast}V\longrightarrow TV.
\]
We say that $\pi$ is {\bf regular} of rank $k$ if its image $\D:=\pi^\sharp(T^{\ast}V)$ is a subbundle of $TV$ of rank $k$. Notice that $\pi$ endows its image $\D$ with an almost symplectic structure $\omega\in\Gamma(V,\wedge^2 \D^{\ast})$, defined by
\[
 \omega(\pi^\sharp(\xi),\pi^\sharp(\eta)):=\pi(\xi,\eta).
\]
Conversely, any subbundle $\D\subset TV$ together with a nondegenerate 2-form $\omega\in\Gamma(V,\wedge^2 \D^{\ast})$ defines a regular bivector $\pi$: it is the unique bivector such that the following diagram commutes:
\[
\xymatrix{
 T^{\ast}V \ar[d]^{i^{\ast}} \ar[r]^{\pi^\sharp} & TV\\
 \D^{\ast} \ar[r]^{\omega^{-1}} & \D \ar[u]^{i}
}
\]
where $i:\D \hookrightarrow TV$ denotes the inclusion and $i^{\ast}$ its transpose.

It is well known that a regular bivector $\pi$ is a {\bf Poisson structure} (i.e., $[\pi,\pi]=0$) if and only if $\D$ is the tangent bundle of a foliation $\mathcal{F}$ and the foliated 2-form $\omega$ is $\d_\F$-closed (in other words, it is a foliated symplectic structure). Hence, Proposition \ref{prop:main} is equivalent to

\begin{theorem}
Let $\pi_0$ be a regular bivector on an open manifold $V$. Then there is a path $t\mapsto\pi_t$ of regular bivectors on $V$ starting at $\pi_0$ and ending at a {\em Poisson} bivector $\pi_1$ if and only if the distribution $\D_0=\im\pi_0^\sharp$ is homotopic to an integrable distribution. Moreover, the Poisson structure $\pi_1$ can be choosen so that its leafwise symplectic form extends to a global closed 2-form.
\end{theorem}

\begin{remark}
According to \cite[Corollary 14]{CF}, if the leafwise symplectic form  of a regular Poisson stucture extends to a global closed 2-form, then its leaves are submanifolds of a very special type, called Lie-Dirac submanifolds, and the Poisson manifold is integrable to a smooth symplectic groupoid. See \cite{CF} for more details.
\end{remark}

Also, we have:

\begin{corollary}
On an open $n$-manifold $V$ satisfying $H^i(V;\mathbb{Z})=0$ for all $i > q+1$, any regular codimension-$q$ bivector is homotopic, through regular bivectors, to a Poisson structure.
\end{corollary}

\section{An Example}
\label{sec:example}

Our main result states that the integrability of a regular bivector is controlled by the integrability of the underlying distribution. However, it could still be the case that the existence of a non-degenerate 2-form on this distribution would force its integrability.  In other words,
one might wonder whether we are any closer to solving the obstruction problem pertaining to integrability of subbundles if we have already solved that of providing an almost symplectic structure (which amounts to lifting the classifying map of the bundle to $B$Sp). In this last section we modify a classical example of Bott to construct an example of a regular bivector on an open manifold which is not homotopic to a regular Poisson structure, showing that not all the integrability obstructions are encoded in the sympletic ones.

Let $E$ be a vector bundle over the manifold $V$ and denote by $\Pont(E)\subset H^{\bullet}_{\text{dR}}(V;\mathbb{R})$ the Pontryagin ring of $E$. Bott's obstruction to integrability of distributions can be stated as follows:

\begin{theorem}[Bott \cite{Bt}]
If a codimension-$q$ distribution $\D$ on $V$ is homotopic to an involutive distribution $T\F$, then its normal bundle $\nu(\D)$ satisfies:
\[
 \Pont^k(\nu(\D))=0,\quad \text{for }k>2q.
\]
\end{theorem}

Recall also that the Pontryagin classes $p_i(E)$ of a (real) vector bundle $E$ are related to the Chern classes $c_i(E\otimes\Cc)$ of its complexification $E\otimes\Cc$ by:
\[
 p_i(E)=(-1)^ic_{2i}(E\otimes\mathbb{C}).
\]
On the other hand, the Euler class and the top Chern class of a complex vector bundle coincide, so the square of the Euler class $e(E)^2=e(E\otimes\mathbb{C})$, lies in the Pontryagin ring of $E$.

In order to construct our example, we start with the trivial complex vector bundle $\underline{\mathbb{C}^{2n}}=\mathbb{C}P^{2n-1}\times\mathbb{C}^{2n}$ over $\mathbb{C}P^{2n-1}$ and split it holomorphically as the sum of the tautological line bundle $\O(-1)$ and its orthogonal :
\[ \underline{\mathbb{C}^{2n}}=\O(-1)\oplus \O(-1)^{\perp},\]
where:
\begin{align*}
\O(-1)&=\{(x,v)\in\mathbb{C}P^{2n-1}\times\mathbb{C}^{2n} : v\in x\},\\
\O(-1)^{\perp}&=\{(x,v)\in\mathbb{C}P^{2n-1}\times\mathbb{C}^{2n} : v \perp x\}.
\end{align*}
We use the usual notation $\O(1)$ for the line bundle dual to $\O(-1)$ and $\O(2)$ for its square $\O(1)\otimes \O(1)$. If we let $\theta_1,...,\theta_{2n}$ be a basis of $(\mathbb{C}^{2n})^{\ast}$ and define on $\mathbb{C}^{2n}$ the non-degenerate 2-form
\[
 \theta=\theta_1\wedge\theta_2+\dots+\theta_{2n-1}\wedge\theta_{2n}.
\]
then we obtain a holomorphic surjection
\[ \theta_{\ast}:\Hom(\O(-1),\O(-1)^{\perp})\longrightarrow \O(2),\quad 
 \theta_{\ast}(\varphi)(\xi)=\theta(\xi,\varphi(\xi)),\]
where $\varphi\in \Hom(\O(-1),\O(-1)^{\perp})$ and $\xi\in \O(-1)$. 

Now, under the well-known isomorphism $T\mathbb{C}P^{2n-1}\simeq \Hom(\O(-1),\O(-1)^{\perp})$ we see that $\D:=\ker(\theta_{\ast})$ is a holomorphic subbundle of (complex) codimension one. If $\nu(\D)$ denotes its normal bundle, then by multiplicativity of the Euler class, one has
\begin{align*}
0\neq 2n=e(T\mathbb{C}P^{2n-1})=e(\D)e(\nu(\D)),
\end{align*}
so that $0 \neq e(\nu(\D))\in H^{2}(\mathbb{C}P^{2n-1};\mathbb{R})$. Thus $e^4(\nu(\D))=e(\nu(\D))^2e(\nu(\D))^2\in \Pont^{8}(\nu(\D))$ as we pointed out above. Since the real cohomology ring of $\mathbb{C}P^{2n-1}$ is the truncated polynomial ring $\mathbb{R}[t]/t^{2n}\mathbb{R}[t]$, where $t$ has degree $2$, choosing $n > 2$ guarantees that $e(\nu(\D))^4\neq 0$. By Bott's result, we conclude that $\D$ is a (real) codimension-2 distribution which is not homotopic to a foliation.

Finally, we consider the open manifold $V:=\mathbb{C}P^{2n-1}\times\mathbb{C}$. The projection in the first factor $p:V\to\mathbb{C}P^{2n-1}$ gives an injection
\[
 p^{\ast}:H^{\bullet}_{\text{dR}}(\mathbb{C}P^{2n-1};\mathbb{R})\to H^{\bullet}_{\text{dR}}(V;\mathbb{R}),
\]
which maps the Euler class of $\nu(\D)$ to that of $p^{\ast}\nu(\D)$, which is the normal bundle to the holomorphic codimension-1 subbundle $\widetilde{\D}:=p^{\ast}\D\oplus\Cc\subset TV$. This means that $\widetilde{\D}$ also cannot be homotopic to a foliation in $V$. A metric on $\widetilde{\D}$, together with its complex structure, yields a non-degenerate 2-form on $\widetilde{\D}$. Hence, this gives an example of a regular bivector on an open manifold which is not homotopic, through regular bivectors, to a Poisson structure.
\vskip 15 pt

\textbf{Acknowledgements.} After reading a preliminary version of this paper, David Martinez Torres has pointed out that our approach can be used to prove foliated versions of the h-principle for conformal symplectic structures, contact structures and, more generally, Jacobi structures (one needs an analogue of Theorem \ref{thm:approx} concerning approximation by $\d_\theta$-closed forms, where $\d_\theta$ denotes the twisted differential). We would like to thank him as well as an anonymous referee for many comments that helped improving this manuscript.

\end{document}